\documentclass[12pt,a4paper]{article}
\usepackage{}
\usepackage{amsfonts}
\usepackage{amssymb}
\usepackage{mathrsfs}
\usepackage{amsmath}
\usepackage{amsthm}
\usepackage{enumerate}
\usepackage{cite}
\usepackage[all]{xy}
\allowdisplaybreaks
\newtheorem{defn}{Definition}[section]
\newtheorem{thm}[defn]{Theorem}
\newtheorem{lem}[defn]{Lemma}
\newtheorem{prop}[defn]{Proposition}
\newtheorem{cor}[defn]{Corollary}
\newtheorem{eg}[defn]{Example}
\newtheorem{re}[defn]{Remark}

\newcommand\relphantom[1]{\mathrel{\phantom{#1}}}

\newcommand{\bdefn}{\begin{defn}}
\newcommand{\edefn}{\end{defn}}
\newcommand{\bthm}{\begin{thm}}
\newcommand{\ethm}{\end{thm}}
\newcommand{\blem}{\begin{lem}}
\newcommand{\elem}{\end{lem}}
\newcommand{\bprop}{\begin{prop}}
\newcommand{\eprop}{\end{prop}}
\newcommand{\bcor}{\begin{cor}}
\newcommand{\ecor}{\end{cor}}
\newcommand{\beg}{\begin{eg}}
\newcommand{\eeg}{\end{eg}}
\newcommand{\bre}{\begin{re}}
\newcommand{\ere}{\end{re}}
\newcommand{\bpf}{\begin{proof}}
\newcommand{\epf}{\end{proof}}

\newcommand{\Chom}{\rm Chom}
\newcommand{\Z}{{\rm Z}}

\newcommand{\QC}{{\rm QC}}
\newcommand{\GDer}{{\rm GDer}}
\newcommand{\ZDer}{{\rm ZDer}}
\newcommand{\Cend}{{\rm Cend}}
\newcommand{\QDer}{{\rm QDer}}
\newcommand{\Der}{{\rm Der}}

\newcommand{\id}{{\rm id}}

\newcommand{\R}{\mathcal{R}}

\newcommand{\A}{\mathscr{A}}

\newcommand{\F}{\mathscr{F}}

\newcommand{\benu}{\begin{enumerate}}
\newcommand{\eenu}{\end{enumerate}}
\newcommand{\bc}{\begin{center}}
\newcommand{\ec}{\end{center}}
\newcommand{\bea}{\begin{eqnarray}}
\newcommand{\eea}{\end{eqnarray}}
\newcommand{\Bea}{\begin{eqnarray*}}
\newcommand{\Eea}{\end{eqnarray*}}
\newcommand{\beq}{\begin{equation}}
\newcommand{\eeq}{\end{equation}}
\newcommand{\Beq}{\begin{equation*}}
\newcommand{\Eeq}{\end{equation*}}
\newcommand{\bspl}{\begin{split}}
\newcommand{\espl}{\end{split}}

\numberwithin{equation}{section}

\begin{document}
\title{{\bf Deformations and generalized derivations of Hom-Lie conformal algebras}}
\author{ Jun Zhao$^{1}$, Lamei Yuan$^{2}$, Liangyun Chen$^{1*}$
 \date{{\small $^{1}$School of Mathematics and Statistics, Northeast Normal
 University,\\
Changchun 130024, China\\
 $^{2}$Department of Mathematics, Harbin Institute of Technology,\\ Harbin 150001, China}}}

\maketitle
\date{}

\begin{abstract}
The purpose of this paper is to extend the cohomology and conformal derivation theories of the classical Lie conformal algebras  to Hom-Lie conformal algebras. In this paper, we develop cohomology theory of Hom-Lie conformal algebras and discuss some applications to the study of deformations of regular Hom-Lie conformal algebras. Also, we introduce $\alpha^k$-derivations of multiplicative Hom-Lie conformal algebras and study their properties.
\bigskip

\end{abstract}
\renewcommand{\thefootnote}{\fnsymbol{footnote}}
\footnote[0]{ Corresponding author(L. Chen): chenly640@nenu.edu.cn.}
\footnote[0]{Supported by NNSF of China (Nos. 11171055, 11471090 and 11301109) and  NSF of  Jilin province (No. 20170101048JC).}

\section{Introduction}
The notion of a Lie conformal algebra encodes an axiomatic description of the operator product expansion of chiral fields in two-dimensional conformal field theory. It has been proved to be an effective tool for the study of infinite-dimensional Lie algebras satisfying the locality property. Besides, vertex operator algebras in \cite{YZ} are closely related to Lie conformal algebras, a vertex operator algebra is an algebraic structure that plays an important role in conformal field theory and string theory. In the last decade, semisimple Lie conformal algebras have been intensively studied. In particular, the cohomology theory was developed in \cite{BKV} and the classification of all finite semisimple Lie conformal algebras were given in \cite{DK}.

Hom-Lie conformal algebras were introduced and studied in \cite{Y}. Lately, similar generalizations of certain algebraic structures became a very popular subject. In \cite{S}, $\alpha^k$-derivations of Hom-Lie algebras were introduced and studied.  In \cite{LCM, ZCM1}, we studied Hom-Nijienhuis operators and $T$*-extensions of Hom-Lie superalgebras and Hom-Jordan Lie algebras, extending the generalized derivation theory of Lie algebras given in \cite{L}. Recently, similar researches were done for Lie conformal algebras in \cite{FHS}. In the present paper, we aim to study generalized derivations of Hom-Lie conformal algebras, and extend the cohomology theory of Lie conformal algebras to the Hom case.

The paper is organized as follows. In Section 2, we recall the notion of a Hom-Lie conformal algebra and then define a module over a Hom-Lie conformal algebra. Moreover, we construct the basic and reduced complexes over a Hom-Lie conformal algebra $\R$ with coefficients in its modules, leading us to the basic and reduced chomologies of $\R$.

In Section 3, we define Hom-Nijienhuis operators of regular Hom-Lie conformal algebras
and show that the deformation generated by a Hom-Nijienhuis operator is trivial.

In Section 4, we study $\alpha^{k}$-derivations of multiplicative Hom-Lie conformal algebras.
Considering the direct sum of the space of $\alpha^k$-derivations, we prove that it is a Hom-Lie conformal algebra. In particular, any $\alpha^1$-derivation gives rise to a derivation extension of a multiplicative Hom-Lie conformal algebra.

In Section 5, we introduce different kinds of generalized derivations of multiplicative Hom-Lie conformal algebras,   and study their properties and connections, extending some results obtained in \cite{L}.

Throughout this paper, all vector
spaces, linear maps, and tensor products are over the complex field $\mathbb{C}$. In addition to the standard
notations $\mathbb{Z}$ and $\mathbb{R}$, we use $\mathbb{Z}_+$ to denote the set of nonnegative integers.
\section{Cohomology of Hom-Lie conformal algebras}

First we present the definition of a Hom-Lie conformal algebra given in \cite{Y}.
\bdefn \rm
A Hom-Lie conformal algebra is a pair $(\R, \alpha)$ in which $\alpha: \R\rightarrow \R$ is a $\mathbb{C}$-linear map satisfying $\alpha\circ\partial=\partial\circ\alpha$, and $\R$ is
 a $\mathbb{C}[\partial]$-module endowed with a $\mathbb{C}$-bilinear map $$[\cdot _\lambda \cdot]:\R\otimes \R\longrightarrow \R[\lambda],\ \  a\otimes b\mapsto [a_\lambda b],$$ called the $\lambda$-bracket, and satisfying the following axioms for $a,b,c\in \R$:
 \begin{eqnarray}
&&{\rm Conformal\ sesquilinearity}:\ \, [\partial a_\lambda b]=-\lambda[a_\lambda b], [a_\lambda\partial b]=(\partial+\lambda)[a_\lambda b];\label{LCF1}\\
 &&{\rm Skew-symmetry}:     \ \ \ \ \ \ \ \ \ \ \ \, [a_\lambda b]=-[b_{-\partial-\lambda}a];\label{LCF2}\\
 &&{\rm Hom\ Jacobi\ identity}:     \ \ \ \ \ \ \ \  \, [\alpha(a)_\lambda[b_\mu c]]=[[a_\lambda b]_{\lambda+\mu}\alpha(c)]+[\alpha(b)_\mu[a_\lambda c]].\label{LCF3}
 \end{eqnarray}
\edefn

As usual in the theory of conformal algebras, the RHS of skew-symmetry means that we have
to take $[b_\mu a]$, expand as a polynomial in $\mu$ with coefficients in $\R$ and then evaluate
$\mu=-\partial-\lambda$ with the corresponding action of $\partial$ in the coefficients.

If we consider the expansion
\begin{eqnarray}
[a_\lambda b]=\mbox{$\sum\limits_{n=0}^{\infty}$} \frac{\lambda^n}{n!} a_{(n)} b,
\end{eqnarray}
the coefficients of $\frac{\lambda^n}{n!}$
are called the $(n)$-products, and the definition can be written in terms of
them (cf. \cite{Y}).

A Hom-Lie conformal algebra $(\R, \alpha)$ is called {\it multiplicative} if $\alpha$ is an algebra endomorphism, i.e.,  $\alpha([a_\lambda b])=[\alpha(a)_\lambda\alpha(b)]$ for any $a,b \in \R$. In particular, if $\alpha$ is an algebra isomorphism, then $(\R, \alpha)$ is called {\it regular}.

In the following, we present a construction of Hom-Lie conformal algebras, extending that of Lie conformal algebras given in \cite{BKV}.

Let $(g,\alpha)$ be a Hom-Lie algebra. A {\it $g$-valued formal distribution} is a series of the form $a(z)=\sum_{n\in\mathbb{Z}}a_n z^{-n-1}$,
where $a_n\in g$ and $z$ is an indeterminate. We denote the space of such distributions by $g[[z,z^{-1}]]$ and the operator $\partial$ on this space by $\partial_z$.
Two $g$-valued formal distributions are called {\it local} if there exists $N\in\mathbb{Z}_+$, such that
\begin{eqnarray*}
(z-w)^N[a(z),b(w)]=0.
\end{eqnarray*}
This is equivalent to saying that one has an expansion of the form:
\begin{eqnarray}
[a(z),b(w)]=\sum_{j=0}^{N-1}(a(w)_{(j)} b(w))\partial_w^{(j)}\delta(z-w),
\end{eqnarray}
where
\begin{eqnarray}a(w)_{(j)}b(w)={\rm Res}_z(z-w)^j[a(z),b(w)]\label{ab}\end{eqnarray}
and $$\delta(z-w)=\sum_{n\in\mathbb{Z}}z^{-n-1}w^n.$$
Let $\F$ be a family of pairwise local $g$-valued formal distributions such that the coefficients of all distributions from $\F$ span $g$. Then the pair $(g,\F)$ is called a {\it formal distribution Hom-Lie algebra}.

Let $\bar{\F}$ denote the minimal subspace of $g[[z.z^{-1}]]$ containing $\F$ which is closed under all $j^{th}$ products \eqref{ab}, $\alpha$-invariant and $\partial$-invariant. One knows that $\bar{\F}$ still consists of pairwise local distributions. Letting
$$[a_\lambda b]=\sum_{n\in\mathbb{Z_+}}\lambda^{(n)}a_{(n)}b,$$
where $\lambda^{(n)}=\lambda^{n}/{n!}$, one endows $\bar{\F}$ with the structure of a Hom-Lie conformal algebra, which is denoted by ${\rm Conf}(g,\F)$.

\bdefn\label{module}\rm
A module $(M,\beta)$ over a Hom-Lie conformal algebra $(\R,\alpha)$ is a $\mathbb{C}[\partial]$-module endowed with a $\mathbb{C}$-linear map $\beta$ and a $\mathbb{C}$-bilinear map $\R\otimes M\longrightarrow M[\lambda]$, $a\otimes v\mapsto a_\lambda v$, such that for $a,b\in \R$, $v\in M$,
 \begin{eqnarray}
&&\alpha(a)_\lambda(b_\mu v)-\alpha(b)_\mu(a_\lambda v)=[a_\lambda b]_{\lambda+\mu}\beta(v),\label{mo1}\\[4pt]
&&(\partial a)_\lambda v=-\lambda(a_\lambda v), \ a_\lambda(\partial v)=(\partial+\lambda)a_\lambda v,\label{mo2}\\[4pt]
&& \beta\circ \partial=\partial\circ \beta,\  \beta(a_\lambda v)=\alpha(a)_\lambda(\beta v).\label{mo3}
 \end{eqnarray}
\edefn

An $\R$-module $(M,\beta)$ is called {\it finite} if it is finitely generated over $\mathbb{C}[\partial]$.

\beg\rm Let $(\R,\alpha)$ be a Hom-Lie conformal algebra.  Then $(\R,\alpha)$ is an $\R$-module under the adjoint diagonal action, namely, $a_\lambda b:=[a_\lambda b]$, $\forall$ $a, b\in \R$.
\eeg
\begin{prop}
Let $(\R,\alpha)$ be a multiplicative Hom-Lie conformal algebra and $(M,\beta)$ an $\R$-module. Define a $\lambda$-bracket $[\cdot_\lambda \cdot]_M$ on $\R \oplus M$ by
\begin{eqnarray*}
[(a+u)_\lambda(b+v)]_M=[a_\lambda b]+a_\lambda v-b_{-\partial-\lambda}u, \ \ \forall\, a,b\in \R,\ u,v\in M.
\end{eqnarray*}
Define $\alpha+\beta:\R\oplus M \rightarrow \R \oplus M$ by
$(\alpha+\beta)(a+u)=\alpha(a)+\beta(u)$.
Then $(\R\oplus M, \alpha+\beta)$ is a multiplicative Hom-Lie conformal algebra.
\end{prop}
\begin{proof}
 Note that $\R \oplus M$ is equipped with a $\mathbb{C}[\partial]$-module structure via
 \begin{eqnarray*}
 \partial(a+u)=\partial(a)+\partial(u),\ \,a\in \R,\, u\in M.
  \end{eqnarray*}
 With this, it is easy to see that $(\alpha+\beta)\circ\partial=\partial\circ(\alpha+\beta)$ and $(\alpha+\beta)\big([(a+u)_\lambda (b+v)]_M\big)=[((\alpha+\beta)(a+u))_\lambda(\alpha+\beta)(b+v)]_M$, $\forall$ $a,b\in \R$, $u,v\in M$.
A direct computation shows that
\begin{eqnarray*}
[\partial(a+u)_\lambda(b+v)]_M
&=&[(\partial a+\partial u)_\lambda(b+v)]_M
=[\partial a_\lambda b]+(\partial a)_\lambda v-b_{-\partial-\lambda}\partial u\\
&=&-\lambda[a_\lambda b]-\lambda a_\lambda v-(\partial-\lambda-\partial)b_{-\partial-\lambda}u\\
&=&-\lambda([a_\lambda b]+a_\lambda v-b_{-\partial-\lambda}u)\\&=&-\lambda[(a+u)_\lambda(b+v)]_M,\\
{[(a+u)_\lambda\partial(b+v)]_M}
&=&[(a+u)_\lambda (\partial b+\partial v)]_M
=[a_\lambda\partial b]+a_\lambda\partial v-(\partial b)_{-\partial-\lambda}u\\
&=&(\partial+\lambda)[a_\lambda b]+(\partial+\lambda)a_\lambda v-(\partial+\lambda)b_{-\partial-\lambda}u\\
&=&(\partial+\lambda)([a_\lambda b]+a_\lambda v-b_{-\partial-\lambda}u)\\
&=&(\partial+\lambda)[(a+u)_\lambda(b+v)]_M.
\end{eqnarray*}
Thus \eqref{LCF1} holds. \eqref{LCF2} follows from
\begin{eqnarray*}
[(b+v)_{-\partial-\lambda}(a+u)]_M
&=&[b_{-\partial-\lambda}a]+b_{-\partial-\lambda}u-a_\lambda v
=-[a_\lambda b]-a_\lambda v+b_{-\partial-\lambda}u\\
&=&-[(a+u)_\lambda(b+v)]_M.
\end{eqnarray*}

To check the Hom Jacobi identity, we compute
\begin{eqnarray}
&&[(\alpha+\beta)(a+u)_\lambda[(b+v)_\mu(c+w)]_M]_M\nonumber\\
&=&[(\alpha(a)+\beta(u))_\lambda[(b+v)_\mu(c+w)]_M]_M\nonumber\\
&=&[(\alpha(a)+\beta(u))_\lambda([b_\mu c]+b_\mu w-c_{-\partial-\mu}v)]_M\nonumber\\
&=&[\alpha(a)_\lambda[b_\mu c]]+\alpha(a)_\lambda (b_\mu w)-\alpha(a)_\lambda(c_{-\partial-\mu}v)-[b_\mu c]_{-\partial-\lambda}(\beta u).\label{1}\\[4pt]
&&[(\alpha+\beta)(b+v)_\mu[(a+u)_\lambda(c+w)]_M]_M\nonumber\\
&=&[\alpha(b)_\mu[a_\lambda c]]+\alpha(b)_\mu (a_\lambda w)-\alpha(b)_\mu(c_{-\partial-\lambda}u)-[a_\lambda c]_{-\partial-\mu}(\beta v).\label{2}\\[4pt]
&&[{[(a+u)_\lambda(b+v)]_M}_{(\lambda+\mu)}(\alpha+\beta)(c+w)]_M\nonumber\\
&=&[([a_\lambda b]+a_\lambda v-b_{-\partial-\lambda}u)_{\lambda+\mu}(\alpha+\beta)(c+w)]_M\nonumber\\
&=&[[a_\lambda b]_{\lambda+\mu}\alpha(c)]+[a_\lambda b]_{\lambda+\mu}(\beta w)-\alpha(c)_{-\partial-\lambda-\mu}(a_\lambda v)
+\alpha(c)_{-\partial-\lambda-\mu}(b_{-\partial-\lambda}u).\label{3}
\end{eqnarray}
By \eqref{1}--\eqref{3}, we only need to show that
\begin{eqnarray}
\alpha(a)_\lambda (b_{-\partial-\mu^{'}}w)-\alpha(b)_{-\partial-\mu^{'}-\lambda}(a_{-\partial-\mu^{'}}w)=[a_\lambda b]_{-\partial-\mu^{'}}(\beta w).\label{represen3}
\end{eqnarray}
Since $(M,\beta)$ is an $\R$-module,
\begin{eqnarray}
\alpha(a)_\lambda (b_\mu w)-\alpha(b)_\mu (a_\lambda w)=[a_\lambda b]_{\lambda+\mu}(\beta w).\label{represen1}
\end{eqnarray}
Replacing $\mu$ by $-\lambda-\mu^{'}-\partial$ in (\ref{represen1}) and using \eqref{mo2}, we obtain
\begin{eqnarray}
\alpha(a)_\lambda (b_{-\partial-\mu^{'}}w)-\alpha(b)_{-\partial-\mu^{'}-\lambda}(a_\lambda w)=[a_\lambda b]_{-\partial-\mu^{'}}(\beta w).\label{represen2}
\end{eqnarray}
By \eqref{mo2} again, (\ref{represen2}) is equivalent to \eqref{represen3}. This implies
\begin{eqnarray*}
&&[(\alpha+\beta)(a+u)_\lambda[(b+v)_\mu(c+w)]_M]_M\\
&=&[{[(a+u)_\lambda(b+v)]_M}_{\lambda+\mu}(\alpha+\beta)(c+w)]_M+
[(\alpha+\beta)(b+v)_\mu[(a+u)_\lambda(c+w)]_M]_M.
\end{eqnarray*}
Therefore $(\R\oplus M, \alpha+\beta)$ is a multiplicative Hom-Lie conformal algebra.
\end{proof}
In the following we aim to develop cohomology theory of Hom-Lie conformal algebras. To do this, we need the following concept.

\bdefn\rm
An $n$-cochain ($n\in \mathbb{Z}_+$) of a multiplicative Hom-Lie conformal algebra $(\R,\alpha)$ with coefficients in a module $(M,\beta)$ is a $\mathbb{C}$-linear map
\begin{eqnarray*}
\gamma:\R^{\otimes n}&\rightarrow& M[\lambda_{1},\cdots,\lambda_{n}],\\
a_1\otimes\cdots\otimes a_n&\mapsto& \gamma_{\lambda_{1},\cdots,\lambda_{n}}(a_1,\cdots,a_n),
\end{eqnarray*}
where $M[\lambda_{1},\cdots,\lambda_{n}]$ denotes the space of polynomials with coefficients in $M$, satisfying the following conditions:\\
Conformal antilinearity: $\gamma_{\lambda_{1},\cdots,\lambda_{n}}(a_1,\cdots,\partial a_i,\cdots,a_n) =-\lambda_{i}\gamma_{\lambda_{1},\cdots,\lambda_{n}}(a_1,\cdots,a_i,\cdots,a_n);$

\noindent Skew-symmetry: $\gamma$ is skew-symmetric with respect to simultaneous permutations of $a_i$'s and $\lambda_i$'s; 

\noindent Commutativity: $\beta\circ\gamma=\gamma\circ\alpha,$ which holds in the sense that
\begin{eqnarray*}
\beta(\gamma(a_1\otimes\cdots\otimes a_n))=\gamma(\alpha(a_1)\otimes\cdots\otimes \alpha(a_n)).
\end{eqnarray*}
\edefn

Let $\R^{\otimes0}=\mathbb{C}$ as usual, so that a $0$-cochain $\gamma$ is an element of $M$. Denote by $\alpha^k$ the $k$-times composition of $\alpha$. Define a differential ${{\rm{\bf d}}}$ of a cochain $\gamma$ by
\begin{eqnarray*}
&&({{\rm{\bf d}}}\gamma)_{\lambda_{1},\cdots,\lambda_{n+1}}(a_1,\cdots,a_{n+1})\\
&=&\mbox{$\sum\limits_{i=1}^{n+1}$}(-1)^{i+1}\alpha^{n}(a_i)_{\lambda_i}\gamma_{\lambda_{1},\cdots,\hat{\lambda_i},
\cdots,\lambda_{n+1}}(a_1,\cdots,\hat{a_i},\cdots,a_{n+1})\\
&&+\mbox{$\sum\limits_{1\leq i<j}^{n+1}(-1)^{i+j}$}\gamma_{\lambda_i+\lambda_j,\lambda_1,\cdots,\hat{\lambda_i},\cdots,\hat{\lambda}_j,\cdots,\lambda_{n+1}}
([{a_i}_{\lambda_i}a_j],\alpha(a_1),\cdots,\hat{a}_i,\cdots,\hat{a}_j,\cdots,\alpha(a_{n+1})),
\end{eqnarray*}
where $\gamma$ is extended linearly over the polynomials in $\lambda_i$. In particular, if $\gamma$ is a $0$-cochain, then $({{\rm{\bf d}}}\gamma)_\lambda a=a_\lambda \gamma$.
\begin{re}\rm
Conformal antilinearity implies the following relation for an $n$-cochain $\gamma$:
$$\gamma_{\lambda+\mu,\lambda_1,\cdots}([a_\lambda b],a_1,\cdots)=\gamma_{\lambda+\mu,\lambda_1,\cdots}([a_{-\partial-\mu} b],a_1,\cdots).$$
\end{re}
\begin{prop} ${{\rm{\bf d}}}\gamma$ is a cochain and ${{\rm{\bf d}}}^{2}=0$.
\end{prop}
\begin{proof} Let $\gamma$ be an $n$-cochain. As discussed in the proof of \cite[Lemma 2.1]{BKV},
${{\rm{\bf d}}}\gamma$ satisfies conformal antilinearity and skew-symmetry. Commutativity is obviously satisfied. Thus
${{\rm{\bf d}}}\gamma$ is an $(n+1)$-cochain.

A straightforward computation shows that
\begin{align}
&({{\rm{\bf d}}}^{2}\gamma)_{\lambda_{1},\cdots,\lambda_{n+2}}(a_1,\cdots,a_{n+2})\notag\\
=&\sum_{i=1}^{n+2}(-1)^{i+1}\alpha^{n+1}(a_i)_{\lambda_i}({{\rm{\bf d}}}\gamma)_{\lambda_{1},\cdots,\hat{\lambda}_i,
\cdots,\lambda_{n+2}}(a_1,\cdots,\hat{a}_i,\cdots,a_{n+2})\notag\\
&+\sum_{1\leq i<j}^{n+2}(-1)^{i+j}({{\rm{\bf d}}}\gamma)_{\lambda_i+\lambda_j,\lambda_1,\cdots,\hat{\lambda}_{i,j},\cdots,\lambda_{n+2}}
([{a_i}_{\lambda_i}a_j],\alpha(a_1),\cdots,\hat{a}_{i,j},\cdots,\alpha(a_{n+2}))\notag\\
=&\sum_{i=1}^{n+2}\sum_{j=1}^{i-1}(-1)^{i+j}\alpha^{n+1}(a_i)_{\lambda_i}
(\alpha^n(a_j)_{\lambda_j}\gamma_{\lambda_{1},\cdots,\hat{\lambda}_{j,i},
\cdots,\lambda_{n+2}}(a_1,\cdots,\hat{a}_{j,i},\cdots,a_{n+2}))\tag{*1}\\
&+\sum_{i=1}^{n+2}\sum_{j=i+1}^{n+2}(-1)^{i+j+1}\alpha^{n+1}(a_i)_{\lambda_i}
(\alpha^n(a_j)_{\lambda_j}\gamma_{\lambda_{1},\cdots,\hat{\lambda}_{i,j},
\cdots,\lambda_{n+2}}(a_1,\cdots,\hat{a}_{i,j},\cdots,a_{n+2}))\tag{*2}\\
&+\sum_{i=1}^{n+2}\sum_{1\leq j<k<i}^{n+2}(-1)^{i+j+k+1}\alpha^{n+1}(a_i)_{\lambda_i}
\gamma_{\lambda_j+\lambda_k,\lambda_1,\cdots,\hat{\lambda}_{j,k,i},\cdots,\lambda_{n+2}}\notag\\
&([{a_j}_{\lambda_j}a_k],\alpha(a_1),\cdots,\hat{a}_{j,k,i},\cdots,\alpha(a_{n+2}))\tag{*3}\\
&+\sum_{i=1}^{n+2}\sum_{1\leq j<i<k}^{n+2}(-1)^{i+j+k}\alpha^{n+1}(a_i)_{\lambda_i}
\gamma_{\lambda_j+\lambda_k,\lambda_1,\cdots,\hat{\lambda}_{j,i,k},\cdots,\lambda_{n+2}}\notag\\
&([{a_j}_{\lambda_j}a_k],\alpha(a_1),\cdots,\hat{a}_{j,i,k},\cdots,\alpha(a_{n+2}))\tag{*4}\\
&+\sum_{i=1}^{n+2}\sum_{1\leq i<j<k}^{n+2}(-1)^{i+j+k+1}\alpha^{n+1}(a_i)_{\lambda_i}
\gamma_{\lambda_j+\lambda_k,\lambda_1,\cdots,\hat{\lambda}_{i,j,k},\cdots,\lambda_{n+2}}\notag\\
&([{a_j}_{\lambda_j}a_k],\alpha(a_1),\cdots,\hat{a}_{i,j,k},\cdots,\alpha(a_{n+2}))\tag{*5}\\
&+\sum_{1\leq i<j}^{n+2}\sum_{k=1}^{i-1}(-1)^{i+j+k}\alpha^{n+1}(a_k)_{\lambda_k}
\gamma_{\lambda_i+\lambda_j,\lambda_1,\cdots,\hat{\lambda}_{k,i,j},\cdots,\lambda_{n+2}}\notag\\
&([{a_i}_{\lambda_i}a_j],\alpha(a_1),\cdots,\hat{a}_{k,i,j},\cdots,\alpha(a_{n+2}))\tag{*6}\\
&+\sum_{1\leq i<j}^{n+2}\sum_{k=i+1}^{j-1}(-1)^{i+j+k+1}\alpha^{n+1}(a_k)_{\lambda_k}
\gamma_{\lambda_i+\lambda_j,\lambda_1,\cdots,\hat{\lambda}_{i,k,j},\cdots,\lambda_{n+2}}\notag\\
&([{a_i}_{\lambda_i}a_j],\alpha(a_1),\cdots,\hat{a}_{i,k,j},\cdots,\alpha(a_{n+2}))\tag{*7}\\
&+\sum_{1\leq i<j}^{n+2}\sum_{k=j+1}^{n+2}(-1)^{i+j+k}\alpha^{n+1}(a_k)_{\lambda_k}
\gamma_{\lambda_i+\lambda_j,\lambda_1,\cdots,\hat{\lambda}_{i,j,k},\cdots,\lambda_{n+2}}\notag\\
&([{a_i}_{\lambda_i}a_j],\alpha(a_1),\cdots,\hat{a}_{i,j,k},\cdots,\alpha(a_{n+2}))\tag{*8}\\
&+\sum_{1\leq i<j}^{n+2}(-1)^{i+j}\alpha^{n}([{a_i}_{\lambda_i}a_j])_{\lambda_i+\lambda_j}
\gamma_{\lambda_{1},\cdots,\hat{\lambda}_j,\cdots,\hat{\lambda}_i,
\cdots,\lambda_{n+2}}(\alpha(a_1),\cdots,\hat{a}_j,\cdots,\hat{a}_i,\cdots,\alpha(a_{n+2}))\tag{*9}\\
&+\sum_{distinct\,i,j,k,l,i<j,k<l}^{n+2}(-1)^{i+j+k+l}sign\{i,j,k,l\}\notag\\
&\gamma_{\lambda_k+\lambda_l,\lambda_i+\lambda_j,\lambda_1,\cdots,\hat{\lambda}_{i,j,k,l},\cdots,\lambda_{n+2}}
(\alpha([{a_k}_{\lambda_k}a_l]),\alpha([{a_i}_{\lambda_i}a_j]),\alpha^{2}(a_1),\cdots,\hat{a}_{i,j,k,l},
\cdots,\alpha^{2}(a_{n+2}))\tag{*10}\\
&+\sum_{i,j,k=1,i<j,k\neq i,j}^{n+2}(-1)^{i+j+k+1}sign\{i,j,k\}\notag\\
&\gamma_{\lambda_i+\lambda_j+\lambda_k,\lambda_1,\cdots,\hat{\lambda}_{i,j,k},\cdots,\lambda_{n+2}}
([[{a_i}_{\lambda_i}a_j]_{\lambda_i+\lambda_j}\alpha(a_k)],\alpha^{2}(a_1),\cdots,\hat{a}_{i,j,k},
\cdots,\alpha^{2}(a_{n+2}))\tag{*11},
\end{align}
where $sign\{i_1,\cdots,i_p\}$ is the sign of the permutation putting the indices in increasing order and $\hat{a}_{i,j,\cdots}$
means that $a_i,a_j,\cdots$ are omitted.

It is obvious that $(^*3)$ and $(^*8)$ summations cancel each other. The same is true for $(^*4)$ and $(^*7)$, $(^*5)$ and $(^*6)$. The Hom Jacobi identity implies $(^*11)=0$, whereas skew-symmetry of $\gamma$ gives $(^*10)=0$. As $M$ is an $\R$-module,
\begin{eqnarray*}
-\alpha(a_i)_{\lambda_i}({a_j}_{\lambda_j}m)+\alpha(a_j)_{\lambda_j}({a_i}_{\lambda_i}m)
+[{a_i}_{\lambda_i}a_j]_{\lambda_i+\lambda_j}(\beta m)=0.
\end{eqnarray*}
By $\beta\circ\gamma=\gamma\circ\alpha,$ $(^*1)$, $(^*2)$ and $(^*9)$ summations cancel.  This proves ${{\rm{\bf d}}}^{2}\gamma=0$.
\end{proof}

Thus the cochains of a multiplicative Hom-Lie conformal algebra $\R$ with coefficients in a module $M$ form a comlex, which is denoted by  $$\widetilde{C}^{\bullet}_{\alpha}=\widetilde{C}^{\bullet}_{\alpha}(\R,M)
=\bigoplus_{n\in\mathbb{Z}_{+}}\widetilde{C}^{n}_{\alpha}(\R,M).$$
This complex is called the \emph{basic complex} for the $\R$-module $(M,\beta)$. 
Moreover, define a (left) $\mathbb{C}[\partial]$-module structure on $\widetilde{C}^{\bullet}_{\alpha}$ by
\begin{eqnarray}
(\partial\gamma)_{\lambda_1,\cdots,\lambda_n}(a_1,\cdots,a_n)=
(\partial_M+\sum_{i=1}^{n}\lambda_i)\gamma_{\lambda_1,\cdots,\lambda_n}(a_1,\cdots,a_n).
\end{eqnarray}
\begin{lem} ${{\rm{\bf d}}}\partial=\partial {{\rm{\bf d}}}$, and therefore the graded subspace $\partial\widetilde{C}^{\bullet}_{\alpha}\subset\widetilde{C}^{\bullet}_{\alpha}$ forms a subcomplex.
\end{lem}
\begin{proof}
For any $\gamma\in\widetilde{C}^{n-1}_{\alpha}(\R,M)$, we have
\begin{eqnarray*}
&&{{\rm{\bf d}}}(\partial\gamma)_{\lambda_1,\cdots,\lambda_n}(a_1,\cdots,a_n)\\
&=&\sum_{i=1}^{n}(-1)^{i+1}\alpha^{n-1}(a_i)_{\lambda_i}(\partial\gamma)_{\lambda_{1},\cdots,\hat{\lambda_i},
\cdots,\lambda_{n}}(a_1,\cdots,\hat{a_i},\cdots,a_{n})\\
&&+\sum_{1\leq i<j}^{n}(-1)^{i+j}(\partial\gamma)_{\lambda_i+\lambda_j,\lambda_1,\cdots,\hat{\lambda_i},\cdots,\hat{\lambda}_j,\cdots,\lambda_{n}}
([{a_i}_{\lambda_i}a_j],\alpha(a_1),\cdots,\hat{a}_i,\cdots,\hat{a}_j,\cdots,\alpha(a_{n}))\\
&=&\sum_{i=1}^{n}(-1)^{i+1}\alpha^{n-1}(a_i)_{\lambda_i}(\partial_M+\sum_{j=1,j\neq i}^{n}\lambda_j)\gamma_{\lambda_{1},\cdots,\hat{\lambda_i},\cdots,\lambda_{n}}(a_1,\cdots,\hat{a_i},\cdots,a_{n})\\
&&+\sum_{1\leq i<j}^{n}(-1)^{i+j}(\partial_M+\sum_{k=1,k\neq i,j}^{n}\lambda_k)\\
&&\gamma_{\lambda_i+\lambda_j,\lambda_1,\cdots,\hat{\lambda_i},\cdots,\hat{\lambda}_j,\cdots,\lambda_{n}}
([{a_i}_{\lambda_i}a_j],\alpha(a_1),\cdots,\hat{a}_i,\cdots,\hat{a}_j,\cdots,\alpha(a_{n}))\\
&=&\sum_{i=1}^{n}(-1)^{i+1}(\partial_M+\sum_{j=1}^{n}\lambda_j)\alpha^{n-1}(a_i)_{\lambda_i}
\gamma_{\lambda_{1},\cdots,\hat{\lambda_i},\cdots,\lambda_{n}}(a_1,\cdots,\hat{a_i},\cdots,a_{n})\\
&&+(\partial_M+\sum_{k=1}^{n}\lambda_k)\sum_{1\leq i<j}^{n}(-1)^{i+j}\\
&&\gamma_{\lambda_i+\lambda_j,\lambda_1,\cdots,\hat{\lambda_i},\cdots,\hat{\lambda}_j,\cdots,\lambda_{n}}
([{a_i}_{\lambda_i}a_j],\alpha(a_1),\cdots,\hat{a}_i,\cdots,\hat{a}_j,\cdots,\alpha(a_{n}))\\
&=&\partial({{\rm{\bf d}}}\gamma)_{\lambda_1,\cdots,\lambda_n}(a_1,\cdots,a_n).
\end{eqnarray*}
So ${{\rm{\bf d}}}\partial=\partial {{\rm{\bf d}}}$ and $\partial\widetilde{C}^{\bullet}_{\alpha}\subset\widetilde{C}^{\bullet}_{\alpha}$ forms a subcomplex.
\end{proof}

Define the quotient complex $$C^{\bullet}_{\alpha}(\R,M)=\widetilde{C}^{\bullet}_{\alpha}(\R,M)/\partial\widetilde{C}^{\bullet}_{\alpha}(\R,M)
=\bigoplus_{n\in\mathbb{Z}_+}C^{n}_{\alpha}(\R,M),$$
called the \emph{reduced complex}.
\bdefn\rm
The \emph{basic cohomology} $\widetilde{\rm H}^{\bullet}_{\alpha}(\R,M)$ of a multiplicative Hom-Lie conformal algebra $\R$ with coefficients
in a module $(M,\beta)$ is the cohomology of the  basic complex $\widetilde{C}^{\bullet}_{\alpha}$. The \emph{reduced cohomology} ${\rm H}^{\bullet}_{\alpha}(\R,M)$ is the cohomology of the reduced complex $C^{\bullet}_{\alpha}$.
\edefn
\begin{re}\rm
The basic cohomology $\widetilde{\rm H}^{\bullet}_{\alpha}(\R,M)$ is naturally a $\mathbb{C}[\partial]$-module, whereas the reduced cohomology ${\rm H}^{\bullet}_{\alpha}(\R,M)$ is a complex vector space.
\end{re}
\begin{re}\rm
The exact sequence $0\rightarrow\partial\widetilde{C}^{\bullet}_{\alpha}\rightarrow\widetilde{C}^{\bullet}_{\alpha}\rightarrow C^{\bullet}_{\alpha}\rightarrow0$ gives the long exact sequence of cohomology:
\begin{eqnarray*}
&&0\rightarrow {\rm H}^{0}_{\alpha}(\partial\widetilde{C}^{\bullet})\rightarrow
\widetilde{\rm H}^{0}_{\alpha}(\R,M)\rightarrow{\rm H}^{0}_{\alpha}(\R,M)\rightarrow\\
&&\ \ \,\rightarrow {\rm H}^{1}_{\alpha}(\partial\widetilde{C}^{\bullet})\rightarrow
\widetilde{\rm H}^{1}_{\alpha}(\R,M)\rightarrow{\rm H}^{1}_{\alpha}(\R,M)\rightarrow\\
&&\ \ \,\rightarrow {\rm H}^{2}_{\alpha}(\partial\widetilde{C}^{\bullet})\rightarrow
\widetilde{\rm H}^{2}_{\alpha}(\R,M)\rightarrow{\rm H}^{2}_{\alpha}(\R,M)\rightarrow\cdots.\\
\end{eqnarray*}
\end{re}
\section{Deformations of Hom-Lie conformal algebras}

Let $(\R,\alpha)$ be a regular Hom-Lie conformal algebra. For any fixed integer $s$, define
\begin{eqnarray}\label{la2}
a_\lambda b=[\alpha^{s}(a)_\lambda b],\ \ \forall \ a,b\in \R.
\end{eqnarray}
\begin{prop} $(\R,\alpha)$ is an $\R$-module with the $\lambda$-action given in \eqref{la2}.
\end{prop}
\begin{proof} It only consists of checking the axioms from Definition \ref{module}.
\end{proof}

\begin{re}\rm In the case of $s=0$, $(\R,\alpha)$ as an $\R$-module is just the usual adjoint module. Otherwise, we denote the module $(\R,\alpha)$ by $\R_s$ to emphasize the dependence of $\R$ on $s$, and call $\R_s$ the $\alpha^s$-adjoint module over $\R$.
\end{re}

Let $\gamma\in\widetilde{C}^{n}_{\alpha}(\R,\R_s)$. Define an operator ${{\rm{\bf d}}}_s:\widetilde{C}^{n}_{\alpha}(\R,\R_s)\rightarrow \widetilde{C}^{n+1}_{\alpha}(\R,\R_s)$ by
\begin{eqnarray*}
&&({{\rm{\bf d}}}_{s}\gamma)_{\lambda_{1},\cdots,\lambda_{n+1}}(a_1,\cdots,a_{n+1})\\
&=&\sum_{i=1}^{n+1}(-1)^{i+1}[\alpha^{n+s}(a_i)_{\lambda_i}\gamma_{\lambda_{1},\cdots,\hat{\lambda_i},
\cdots,\lambda_{n+1}}(a_1,\cdots,\hat{a_i},\cdots,a_{n+1})]\\
&&+\sum_{1\leq i<j}^{n+1}(-1)^{i+j}\gamma_{\lambda_i+\lambda_j,\lambda_1,\cdots,\hat{\lambda_i},\cdots,\hat{\lambda}_j,\cdots,\lambda_{n+1}}
([{a_i}_{\lambda_i}a_j],\alpha(a_1),\cdots,\hat{a}_i,\cdots,\hat{a}_j,\cdots,\alpha(a_{n+1})).
\end{eqnarray*}
Obviously, the operator ${{\rm{\bf d}}}_s$ is induced from the differential ${{\rm{\bf d}}}$. Thus ${{\rm{\bf d}}}_s$ preserves the space of cochains and satisfies ${{\rm{\bf d}}}_s^2=0$. In the following the complex $\widetilde{C}^{\bullet}_{\alpha}(\R, \R_s)$ is assumed to be associated with the differential ${{\rm{\bf d}}}_s$.

Taking $s=-1$, for $\psi\in\widetilde{C}^{2}_{\alpha}(\R,\R_{-1})$, we consider a $t$-parameterized family of bilinear operations on $\R$
\begin{eqnarray}\label{1-3}
[a_\lambda b]_t=[a_\lambda b]+t\psi_{\lambda,-\partial-\lambda}(a,b), \ \forall \, a,b\in \R.
\end{eqnarray}
Since $\psi$ commutes with $\alpha$, $\alpha$ is an algebra homomorphism with respect to the bracket $[\cdot_{\lambda} \cdot]_{t}$ for each $t$. If $[\cdot_{\lambda}\cdot]_{t}$ endows $(\R, [\cdot_{\lambda}\cdot]_{t}, \alpha)$ with a regular Hom-Lie conformal algebra structure, we say that $\psi$ generates a deformation of the regular Hom-Lie conformal algebra $(\R,\alpha)$. It is easy to see that $[\cdot_{\lambda}\cdot]_{t}$ satisfies \eqref{LCF1} and \eqref{LCF2}. If it is true for \eqref{LCF3}, expanding the Hom Jacobi identity for $[\cdot_{\lambda}\cdot]_{t}$ gives
\begin{eqnarray*}
&&[\alpha(a)_\lambda[b_\mu c]]+t([\alpha(a)_\lambda\psi_{\mu,-\partial-\mu}(b,c)]
+\psi_{\lambda,-\partial-\lambda}(\alpha(a),[b_\mu c]))\\
&&+t^{2}\psi_{\lambda,-\partial-\lambda}(\alpha(a),\psi_{\mu,-\partial-\mu}(b,c))\\
&=&[\alpha(b)_\mu[a_\lambda c]]+t([\alpha(b)_\mu\psi_{\lambda,-\partial-\lambda}(a,c)]
+\psi_{\mu,-\partial-\mu}(\alpha(b),[a_\lambda c]))\\
&&+t^{2}\psi_{\mu,-\partial-\mu}(\alpha(b),\psi_{\lambda,-\partial-\lambda}(a,c))\\
&&+[[a_\lambda b]_{\lambda+\mu}\alpha(c)]+t([\psi_{\lambda,-\partial-\lambda}(a,b)_{\lambda+\mu}\alpha(c)]+\psi_{\lambda+\mu,-\partial-\lambda-\mu}([a_\lambda b],\alpha(c)))\\
&&+t^{2}\psi_{\lambda+\mu,-\partial-\lambda-\mu}(\psi_{\lambda,-\partial-\lambda}(a,b),\alpha(c)).
\end{eqnarray*}
This is equivalent to the following relations
\begin{eqnarray}
&&[\alpha(a)_\lambda\psi_{\mu,-\partial-\mu}(b,c)]
+\psi_{\lambda,-\partial-\lambda}(\alpha(a),[b_\mu c])-\psi_{\lambda+\mu,-\partial-\lambda-\mu}([a_\lambda b],\alpha(c))\notag\\
&&\ \ \ \, =[\alpha(b)_\mu\psi_{\lambda,-\partial-\lambda}(a,c)]
+\psi_{\mu,-\partial-\mu}(\alpha(b),[a_\lambda c])
+[\psi_{\lambda,-\partial-\lambda}(a,b)_{\lambda+\mu}\alpha(c)] \label{represen9}\\
&&\psi_{\lambda,-\partial-\lambda}(\alpha(a),\psi_{\mu,-\partial-\mu}(b,c))\notag\\
&&\ \ \ \, =\psi_{\mu,-\partial-\mu}(\alpha(b),\psi_{\lambda,-\partial-\lambda}(a,c))
+\psi_{\lambda+\mu,-\partial-\lambda-\mu}(\psi_{\lambda,-\partial-\lambda}(a,b),\alpha(c)).\label{represen10}
\end{eqnarray}
By skew-symmetry and conformal sesquilinearity of $\psi$, we have
\begin{align}[\psi_{\lambda,-\partial-\lambda}(a,b)_{\lambda+\mu}\alpha(c)]
=-[\alpha(c)_{-\partial-\lambda-\mu}\psi_{\lambda,-\partial-\lambda}(a,b)]
=-[\alpha(c)_{-\partial-\lambda-\mu}\psi_{\lambda,\mu}(a,b)].\label{represen11}
\end{align}

On the other hand, let $\psi$ be a cocycle, i.e., ${{\rm{\bf d}}}_{-1}\psi=0.$ Explicitly,
\begin{eqnarray}
0&=&({{\rm{\bf d}}}_{-1}\psi)_{\lambda,\mu,\gamma}(a,b,c)\notag\\
&=&[\alpha(a)_\lambda\psi_{\mu,\gamma}(b,c)]-[\alpha(b)_\mu\psi_{\lambda,\gamma}(a,c)]
+[\alpha(c)_{\gamma}\psi_{\lambda,\mu}(a,b)]\notag\\
&&-\psi_{\lambda+\mu,\gamma}([a_\lambda b],\alpha(c))+\psi_{\lambda+\gamma,\mu}([a_\lambda c],\alpha(b))
-\psi_{\mu+\gamma,\lambda}([b_\mu c],\alpha(a))\notag\\
&=&[\alpha(a)_\lambda\psi_{\mu,\gamma}(b,c)]-[\alpha(b)_\mu\psi_{\lambda,\gamma}(a,c)]
-[\psi_{\lambda,\mu}(a,b)_{-\partial-\gamma}\alpha(c)]\notag\\
&&+\psi_{\lambda,\mu+\gamma}(\alpha(a),[b_\mu c])-\psi_{\mu,\lambda+\gamma}(\alpha(b),[a_\lambda c])
-\psi_{\lambda+\mu,\gamma}([a_\lambda b],\alpha(c)). \label{represen12}
\end{eqnarray}
By \eqref{LCF1}, \eqref{represen11} and replacing $\gamma$ by $-\lambda-\mu-\partial$ in (\ref{represen12}), we obtain
\begin{eqnarray*}
0&=&[\alpha(a)_\lambda\psi_{\mu,-\partial-\mu}(b,c)]-[\alpha(b)_\mu\psi_{\lambda,-\partial-\lambda}(a,c)]
-[\psi_{\lambda,\mu}(a,b)_{\lambda+\mu}\alpha(c)]\\
&&+\psi_{\lambda,-\partial-\lambda}(\alpha(a),[b_\mu c])-\psi_{\mu,-\partial-\mu}(\alpha(b),[a_\lambda c])
-\psi_{\lambda+\mu,-\partial-\lambda-\mu}([a_\lambda b],\alpha(c)),
\end{eqnarray*}
which is exactly \eqref{represen9}. Thus, when $\psi$ is a 2-cocycle satisfying \eqref{represen10},
 $(\R, [\cdot_{\lambda}\cdot]_{t}, \alpha)$ forms a regular Hom-Lie conformal algebra. In this case,
$\psi$ generates a deformation of the regular Hom-Lie conformal algebra $(\R,\alpha)$.

A deformation is said to be {\it trivial} if there is a linear operator $f\in\widetilde{C}^{1}_{\alpha}(\R,\R_{-1})$ such that for
${T_t}_\lambda={\rm id}+tf_\lambda$, there holds
\begin{eqnarray}\label{1-5}
{T_t}_{-\partial}([a_\lambda b]_t)=[({{T_t}_\lambda(a)})_\lambda {T_t}_{-\partial}(b)], \ \forall \ a,b\in\R.
\end{eqnarray}
\begin{defn}\rm
A linear operator $f\in\widetilde{C}^{1}_{\alpha}(\R,\R_{-1})$ is called a Hom-Nijienhuis operator if
\begin{eqnarray}\label{1-4}
[({f_\lambda(a)})_\lambda (f_\mu(b))]=f_{\lambda+\mu}([a_\lambda b]_N), \ \forall \ a,b\in\R,
\end{eqnarray}
where the bracket $[\cdot_{\lambda}\cdot]_N$ is defined by
\begin{eqnarray}\label{N1}
[a_\lambda b]_N=[(f_\lambda(a))_\lambda b]+[a_\lambda(f_{-\partial}(b))]-f_{-\partial}([a_\lambda b]), \ \forall \ a,b\in\R.\end{eqnarray}
\end{defn}
\begin{re}\rm
In particular, by $(\rm C1)_\lambda$ and setting $\mu=-\partial-\lambda$ in Eq.\eqref{1-4}, we obtain
\begin{eqnarray}\label{4-10}
[{(f_\lambda(a))}_\lambda f_{-\partial}(b)]=f_{-\partial}([a_\lambda b]_N), \ \forall \ a,b\in\R.
\end{eqnarray}
\end{re}
\begin{thm} Let $(\R,\alpha)$ be a regular Hom-Lie conformal algebra, and $f\in \tilde{C}_{\alpha}^{1}(\R,\R_{-1})$ a Hom-Nijienhuis operator. Then a deformation of $(\R,\alpha)$ can be obtained by putting
\begin{eqnarray}\label{N2}
\psi_{\lambda,-\partial-\lambda}(a, b):=({{\rm{\bf d}}}_{-1}f)_{\lambda,-\partial-\lambda}(a, b):=[a_\lambda b]_{N},  \ \forall \ a,b\in\R.
\end{eqnarray}
Furthermore,  this deformation is trivial.
\end{thm}
\begin{proof}
Since $\psi={{\rm{\bf d}}}_{-1}f$,  ${{\rm{\bf d}}}_{-1}\psi=0$ is valid. To see that $\psi$ generates a deformation of $(\R,\alpha)$,  we need to check \eqref{represen10} for $\psi$. By \eqref{N1} and \eqref{N2}, we get
\begin{eqnarray*}
\psi_{\lambda,-\partial-\lambda}(\alpha(a), \psi_{\mu,-\partial-\mu}(b,c))=[\alpha(a)_\lambda[b_\mu c]_N]_N,
\end{eqnarray*}
where the right hand side reads
\begin{eqnarray*}
&&\psi_{\lambda,-\partial-\lambda}(\alpha(a), \psi_{\mu,-\partial-\mu}(b,c))=[\alpha(a)_\lambda[b_\mu c]_N]_N\\
&=&[(f_\lambda\alpha(a))_\lambda([b_\mu c]_N)]+[\alpha(a)_\lambda(f_{-\partial}([b_\mu c]_N))]-f_{-\partial}([\alpha(a)_\lambda[b_\mu c]_N])\\
&=&[(f_\lambda\alpha(a))_\lambda[(f_\mu(b))_\mu c]]+[(f_\lambda\alpha(a))_\lambda[b_\mu(f_{-\partial}(c))]]
-[(f_\lambda\alpha(a))_\lambda(f_{-\partial}([b_\mu c]))]\\
&&+[\alpha(a)_\lambda(f_{-\partial}([b_\mu c]_N))]\\
&&-f_{-\partial}([\alpha(a)_\lambda[(f_\mu(b))_\mu c]])-f_{-\partial}([\alpha(a)_\lambda[b_\mu(f_{-\partial}(c))]])+f_{-\partial}([\alpha(a)_\lambda(f_{-\partial}([b_\mu c]))])\\
&=&\underbrace{[(f_\lambda\alpha(a))_\lambda[(f_\mu(b))_\mu c]]}_{(1)}+\underbrace{[(f_\lambda\alpha(a))_\lambda[b_\mu(f_{-\partial}(c))]]}_{(2)}
-[(f_\lambda\alpha(a))_\lambda(f_{-\partial}([b_\mu c]))]\\
&&+\underbrace{[\alpha(a)_\lambda[(f_\mu(b))_\mu(f_{-\partial}(c))]]}_{(3)}\\
&&\underbrace{-f_{-\partial}([\alpha(a)_\lambda[(f_\mu(b))_\mu c]])}_{(4)}\underbrace{-f_{-\partial}([\alpha(a)_\lambda[b_\mu (f_{-\partial}(c))]])}_{(5)}+f_{-\partial}([\alpha(a)_\lambda(f_{-\partial}([b_\mu c]))]).
\end{eqnarray*}
Therefore,
\begin{eqnarray*}
&&\psi_{\mu,-\partial-\mu}(\alpha(b), \psi_{\lambda,-\partial-\lambda}(a,c))\\
&=&\underbrace{[(f_\mu\alpha(b))_\mu[(f_\lambda(a))_\lambda c]]}_{(1)'}+\underbrace{[(f_\mu\alpha(b))_\mu[a_\lambda (f_{-\partial}(c))]]}_{(3)'}-[(f_\mu\alpha(b))_\mu(f_{-\partial}([a_\lambda c]))]\\
&&+\underbrace{[\alpha(b)_\mu[(f_\lambda(a))_\lambda(f_{-\partial}(c))]]}_{(2)'}\underbrace{-f_{-\partial}([\alpha(b)_\mu[(f_\lambda(a))_\lambda c]])}_{(6)'}\underbrace{-f_{-\partial}([\alpha(b)_\mu[a_\lambda (f_{-\partial}(c))]])}_{(5)'}\\
&&+f_{-\partial}([\alpha(b)_\mu(f_{-\partial}([a_\lambda c]))])
\end{eqnarray*}
and
\begin{eqnarray*}
&&\psi_{\lambda+\mu,-\partial-\lambda-\mu}(\psi_{\lambda,-\partial-\lambda}(a,b),\alpha(c))\\
&=&[(f_{\lambda+\mu}([a_\lambda b]_N))_{\lambda+\mu}\alpha(c)]+[{([a_\lambda b]_N)}_{\lambda+\mu}(f_{-\partial}\alpha(c))]-f_{-\partial}([{([a_\lambda b]_N)}_{\lambda+\mu)}\alpha(c)])\\
&=&\underbrace{[[(f_{\lambda}(a))_\lambda(f_{\mu}(b))]_{\lambda+\mu}\alpha(c)]}_{(1)''}\\
&&+\underbrace{[[(f_\lambda(a))_\lambda b]_{\lambda+\mu}(f_{-\partial}\alpha(c))]}_{(2)''}+\underbrace{[[a_\lambda (f_{-\partial}(b))]_{\lambda+\mu}(f_{-\partial}\alpha(c))]}_{(3)''}-[(f_{-\partial}([a_\lambda b]))_{\lambda+\mu}(f_{-\partial}\alpha(c))]\\
&&\underbrace{-f_{-\partial}([(f_\lambda(a))_\lambda b]_{\lambda+\mu}\alpha(c)])}_{(6)''}\underbrace{-f_{-\partial}([[a_\lambda (f_{-\partial}(b))]_{\lambda+\mu}}_{(4)''}\alpha(c)])+f_{-\partial}([(f_{-\partial}([a_\lambda b]))_{\lambda+\mu}\alpha(c)]).
\end{eqnarray*}
Since $f$ is a Hom-Nijienhuis operator, we get
\begin{eqnarray*}
&&\ \ \ \ \,-[(f_\lambda\alpha(a))_\lambda(f_{-\partial}([b_\mu c]))]+f_{-\partial}([\alpha(a)_\lambda(f_{-\partial}([b_\mu c]))])\\
&&\ \ \ \ \ \ \ \ \ \ \ \ \ \ \ \ \ \ \ \,=\underbrace{-f_{-\partial}([(f_\lambda\alpha(a))_\lambda[b_\mu c]])}_{(6)}+\underbrace{f_{-\partial}^2([\alpha(a)_\lambda [b_\mu c]])}_{(7)},\\
&&\ \ \ \ \,-[(f_\mu\alpha(b))_\mu(f_{-\partial}([a_\lambda c]))]+f_{-\partial}([\alpha(b)_\mu(f_{-\partial}([a_\lambda c]))])\\
&&\ \ \ \ \ \ \ \ \ \ \ \ \ \ \ \ \ \ \ \,=\underbrace{-f_{-\partial}([(f_\mu\alpha(b))_\mu [a_\lambda c]])}_{(4)^{'}}
+\underbrace{f_{-\partial}^2([\alpha(b)_\mu [a_\lambda c]])}_{(7)^{'}}.
\end{eqnarray*}
By \eqref{LCF1} and \eqref{4-10},
\begin{eqnarray*}
&&-[(f_{-\partial}([a_\lambda b]))_{\lambda+\mu}(f_{-\partial}\alpha(c))]+f_{-\partial}([(f_{-\partial}([a_\lambda b]))_{\lambda+\mu}\alpha(c)])\\
&=&-[(f_{\lambda+\mu}([a_\lambda b]))_{\lambda+\mu}(f_{-\partial}\alpha(c))]+f_{-\partial}([(f_{\lambda+\mu}([a_\lambda b]))_{\lambda+\mu}\alpha(c)])\\
&=&\underbrace{-f_{-\partial}([[a_\lambda b]_{\lambda+\mu}(f_{-\partial}\alpha(c))])}_{(5)^{''}}+\underbrace{f_{-\partial}^2([[a_\lambda b]_{\lambda+\mu}\alpha(c)])}_{(7)^{''}}.
\end{eqnarray*}
Note that according to the Hom Jacobi identity and \eqref{LCF1} for $a ,b , c\in\R$,
$$[\alpha(a)_\lambda[(f_\mu(b))_\mu(f_{-\partial}(c))]]=[[a_\lambda f_\mu(b)]_{\lambda+\mu}(f_{-\partial}\alpha(c))]+[(f_\mu\alpha(b))_\mu[a_\lambda (f_{-\partial}(c))]]$$
is equivalent to
\begin{eqnarray*}[\alpha(a)_\lambda[(f_\mu(b))_\mu(f_{-\partial}(c))]]=[[a_\lambda f_{-\partial}(b)]_{\lambda+\mu}(f_{-\partial}\alpha(c))]+[(f_\mu\alpha(b))_\mu[a_\lambda(f_{-\partial}(c))]].\end{eqnarray*}
Thus $(i)+(i)^{'}+(i)^{''}=0$, for $i=1,\cdots,7$. This proves that
$\psi$ generates a deformation of the regular Hom-Lie conformal algebra $(\R,\alpha)$.

Let ${T_{t}}_\lambda=\id+tf_\lambda$. By \eqref{1-3} and \eqref{N2},
\begin{eqnarray}\label{1-1}
{T_{t}}_{-\partial}([a_\lambda b]_{t})
&=&(\id+tf_{-\partial})([a_\lambda b]+t\psi_{\lambda,-\partial-\lambda}(a,b))\nonumber\\
&=&(\id+tf_{-\partial})([a_\lambda b]+t[a_\lambda b]_{N})\nonumber\\
&=&[a_\lambda b]+t([a_\lambda b]_{N}+f_{-\partial}([a_\lambda b]))+t^{2}f_{-\partial}([a_\lambda b]_{N}).
\end{eqnarray}
On the other hand,
\begin{eqnarray}\label{1-2}
[{T_{t}}_\lambda(a)_\lambda {T_{t}}_{-\partial}(b)]
&=&[(a+tf_\lambda(a))_\lambda(b+tf_{-\partial}(b))]\nonumber\\
&=&[a_\lambda b]+t([(f_\lambda(a))_\lambda b]+[a_\lambda (f_{-\partial}(b))])+t^{2}[(f_\lambda(a))_\lambda (f_{-\partial}(b))].
\end{eqnarray}
Combining \eqref{1-1} with \eqref{1-2} gives
$T_{t}([a_\lambda b]_{t})=[T_{t}(a)_\lambda T_{t}(b)].$
Therefore the deformation is trivial.
\end{proof}
\section{$\alpha^k$-Derivations of multiplicative Hom-Lie conformal algebras}

For convenience, we denote by $\A$ the ring $\mathbb{C}[\partial]$ of polynomials in the indeterminate $\partial$.
\begin{defn}\rm
A conformal linear map between $\A$-module $V$ and $W$ is a linear map $\phi:V\rightarrow \A[\lambda]\otimes_{\A}W $ such that
\begin{eqnarray}\label{clm}
\phi(\partial v)=(\partial+\lambda)(\phi v).
\end{eqnarray}
\end{defn}

We will often abuse the notation by writing $\phi:V\rightarrow W$ any time it is clear from the context that $\phi$ is a conformal linear map. We will also write $\phi_\lambda$ instead of $\phi$ to emphasize the dependence of $\phi$ on $\lambda$.

The set of all conformal linear maps from $V$ to $W$ is denoted by \Chom$(V,W)$ and is made into an $\A$-module via
\begin{eqnarray}\label{der1}
(\partial\phi)_\lambda v=-\lambda\phi_\lambda v.
\end{eqnarray}
We will write \Cend$(V)$ for \Chom$(V,V)$.
\begin{defn}\rm
Let $(\R,\alpha)$ be a multiplicative Hom-Lie conformal algebra. Then a conformal linear map $D_\lambda:\R\rightarrow \R[\lambda]$ is called an $\alpha^{k}$-derivation of $(\R,\alpha)$ if
\begin{eqnarray*} D_\lambda\circ\alpha=\alpha\circ{D_\lambda}, \end{eqnarray*}
and
\begin{eqnarray}D_\lambda([a_\mu b])=[D_\lambda(a)_{\lambda+\mu}\alpha^{k}(b)]+[\alpha^{k}(a)_\mu D_\lambda(b)].\end{eqnarray}
\end{defn}

Denote by $ \texttt{\Der}_{\alpha^s}(\R)$ the set of $\alpha^s$-derivations of the multiplicative  Hom-Lie conformal algebra $(\R,\alpha)$. For any $a\in \R$ satisfying $\alpha(a)=a$, define $D_{k}(a):\R \rightarrow \R$ by\\
$$D_{k}(a)_\lambda(b)=[a_\lambda\alpha^{k}(b)], \quad \forall \, b\in \R.$$
Then $D_{k}(a)$ is an $\alpha^{k+1}$-derivation, which is called an \textbf{inner} $\alpha^{k+1}$-derivation. In fact,
\begin{eqnarray*}
D_{k}(a)_\lambda(\partial b)&=&[a_\lambda\alpha^{k}(\partial b)]=[a_\lambda\partial(\alpha^{k}(b))]=(\partial+\lambda) D_{k}(a)_\lambda(b),\\
D_{k}(a)_\lambda(\alpha(b))&=&[a_\lambda\alpha^{k+1}(b)]=\alpha([a_\lambda\alpha^{k}(b)])=\alpha\circ D_{k}(a)_\lambda(b), \\
D_{k}(a)_\lambda([b_\mu c])
&=&[a_\lambda\alpha^{k}([b_\mu c])]
=[\alpha(a)_\lambda[\alpha^{k}(b)_\mu\alpha^{k}(c)]]\\
&=&[\alpha^{k+1}(b)_\mu[a_\lambda\alpha^{k}(c)]]+[[a_\lambda\alpha^{k}(b)]_{\lambda+\mu}\alpha^{k+1}(c)]\\
&=&[(D_{k}(a)_\lambda(b))_{\lambda+\mu}\alpha^{k+1}(c)]+[\alpha^{k+1}(b)_\mu(D_{k}(a)_\lambda(c))].
\end{eqnarray*}
Denote by $ \mathrm{Inn}_{\alpha^k}(\R)$ the set of inner  $\alpha^{k}$-derivations. 
For $D_\lambda \in \texttt{\Der}_{\alpha^{k}}(\R)$ and $D^{'}_{\mu-\lambda} \in \texttt{\Der}_{\alpha^{s}}(\R),$ define their commutator $[D_\lambda D^{'}]_\mu$ by
\begin{eqnarray}\label{anticom}
[D_\lambda D^{'}]_\mu(a)=D_\lambda(D^{'}_{\mu-\lambda}a)-D^{'}_{\mu-\lambda}(D_\lambda a), \ \forall\, a\in\R.
\end{eqnarray}
\begin{lem}For any $D_\lambda \in \texttt{\Der}_{\alpha^{k}}(\R)$ and $D^{'}_{\mu-\lambda} \in \texttt{\Der}_{\alpha^{s}}(\R)$, we have
$$[D_\lambda D^{'}]\in \texttt{\Der}_{\alpha^{k+s}}(\R)[\lambda].$$
\end{lem}
\begin{proof}
For any $a,b \in \R$, we have
\begin{eqnarray*}
[D_\lambda D^{'}]_\mu (\partial a)
&=&D_\lambda(D^{'}_{\mu-\lambda}\partial a)-D^{'}_{\mu-\lambda}(D_\lambda \partial a)\\
&=&D_\lambda((\partial+\mu-\lambda)D^{'}_{\mu-\lambda}a)+ D^{'}_{\mu-\lambda}((\partial+\lambda)D_\lambda a)\\
&=&(\partial+\mu)D_\lambda(D^{'}_{\mu-\lambda}a)-(\partial+\mu)D^{'}_{\mu-\lambda}(D_\lambda a)\\
&=&(\partial+\mu)[D_\lambda D^{'}]_\mu(a),
\end{eqnarray*}
and
\begin{eqnarray*}[D_\lambda D^{'}]_\mu([a_\gamma b])
&=&D_\lambda(D^{'}_{\mu-\lambda}[a_\gamma b])-D^{'}_{\mu-\lambda}(D_\lambda [a_\gamma b])\\
&=&D_\lambda({[D^{'}_{\mu-\lambda}(a)}_{\mu-\lambda+\gamma}\alpha^{s}(b)]+[\alpha^{s}(a)_\gamma D^{'}_{\mu-\lambda}(b)])\\
&&-D^{'}_{\mu-\lambda}([D_\lambda(a)_{\lambda+\gamma}\alpha^{k}(b)]+[\alpha^{k}(a)_\gamma D_\lambda(b)])\\
&=&[(D_\lambda( D^{'}_{\mu-\lambda}(a))_{\mu+\gamma}\alpha^{k+s}(b)]+[\alpha^{k}(D^{'}_{\mu-\lambda}(a))_{\mu-\lambda+\gamma} D_\lambda(\alpha^{s}(b))]\\
&&+[D_\lambda(\alpha^{s}(a))_{\lambda+\gamma}\alpha^{k}(D^{'}_{\mu-\lambda}(b))]+[\alpha^{k+s}(a)_\gamma (D_\lambda( D^{'}_{\mu-\lambda}(b)))]\\
&&-[(D^{'}_{\mu-\lambda}D_\lambda(a))_{\mu+\gamma}\alpha^{k+s}(b)]-[\alpha^{s}(D_\lambda(a))_{\lambda+\gamma}( D^{'}_{\mu-\lambda}(\alpha^{k}(b)))]\\
&&-[(D^{'}_{\mu-\lambda}(\alpha^{k}(a)))_{\mu-\lambda+\gamma}\alpha^{s}(D_\lambda(b))]-[\alpha^{k+s}(a)_\lambda(D ^{'}_{\mu-\lambda} (D_\lambda(b)))]\\
&=&[([D_\lambda D^{'}]_\mu a)_{\mu+\gamma}\alpha^{k+s}(b)]+[\alpha^{k+s}(a)_\gamma([D_\lambda D^{'}]_\mu b)].
\end{eqnarray*}
Therefore, $[D_\lambda D^{'}]\in \texttt{\Der}_{\alpha^{k+s}}(\R)[\lambda]$.
\end{proof}
Denote
\begin{eqnarray}
\texttt{\Der}(\R)=\bigoplus_{k\geq 0} \texttt{\Der}_{\alpha^{k}}(\R).\end{eqnarray}

\begin{prop} $(\texttt{\Der}(\R),\alpha^{'})$ is a Hom-Lie conformal algebra with respect to \eqref{anticom}, where
$\alpha^{'}(D)=D\circ\alpha$.
\end{prop}
\begin{proof} By \eqref{der1}, $\texttt{\Der}(\R)$ is a $\mathbb{C}[\partial]$-module.
By \eqref{clm}, \eqref{der1} and \eqref{anticom}, \eqref{LCF1} and \eqref{LCF2} are satisfied.
To check the Hom Jacobi identity, we compute separately
\begin{eqnarray*}
[\alpha^{'}(D)_\lambda[D^{'}_\mu D^{''}]]_\theta(a)
&=&(D\circ\alpha)_\lambda([D^{'}_\mu D^{''}]_{\theta-\lambda} a)-[D^{'}_\mu D^{''}]_{\theta-\lambda}((D\circ\alpha)_\lambda a)\\
&=&D_\lambda([D^{'}_\mu D^{''}]_{\theta-\lambda}\alpha(a))-[D^{'}_\mu D^{''}]_{\theta-\lambda}(D_\lambda\alpha(a))\\
&=&D_\lambda(D^{'}_\mu(D^{''}_{\theta-\lambda-\mu}\alpha(a)))-D_\lambda(D^{''}_{\theta-\lambda-\mu}(D^{'}_\mu \alpha(a)))\\
&&-D^{'}_\mu(D^{''}_{\theta-\lambda-\mu}(D_\lambda\alpha(a)))+D^{''}_{\theta-\lambda-\mu}(D^{'}_\mu(D_\lambda\alpha(a))),\\
{[\alpha^{'}(D^{'})_\mu[D_\lambda D^{''}]]_\theta(a)}
&=&D^{'}_\mu(D_\lambda(D^{''}_{\theta-\lambda-\mu}\alpha(a)))-D^{'}_\mu(D^{''}_{\theta-\lambda-\mu}(D_\lambda\alpha(a)))\\
&&-D_\lambda(D^{''}_{\theta-\lambda-\mu}(D^{'}_\mu\alpha(a)))+D^{''}_{\theta-\lambda-\mu}(D_\lambda(D^{'}_\mu\alpha(a))),\\
{[[D_\lambda D^{'}]_{\lambda+\mu}\alpha^{'}(D^{''})]_\theta(a)}
&=&[D_\lambda D^{'}]_{\lambda+\mu}(D^{''}_{\theta-\lambda-\mu}\alpha(a))-D^{''}_{\theta-\lambda-\mu}([D_\lambda D^{'}]_{\lambda+\mu}\alpha(a))\\
&=&D_\lambda(D^{'}_\mu(D^{''}_{\theta-\lambda-\mu}\alpha(a)))-D^{'}_\mu(D_\lambda(D^{''}_{\theta-\lambda-\mu}\alpha(a)))\\
&&-D^{''}_{\theta-\lambda-\mu}(D_\lambda(D^{'}_\mu\alpha(a)))+D^{''}_{\theta-\lambda-\mu}(D^{'}_\mu(D_\lambda\alpha(a))).
\end{eqnarray*}
 Thus $[\alpha^{'}(D)_\lambda[D^{'}_\mu D^{''}]]_\theta(a)=[\alpha^{'}(D^{'})_\mu[D_\lambda D^{''}]]_\theta(a)+[[D_\lambda D^{'}]_{\lambda+\mu}\alpha^{'}(D^{''})]_\theta(a).$ This proves that $(\texttt{\Der}(\R),\alpha^{'})$ is a Hom-Lie conformal algebra.
\end{proof}


At the end of this section,  we give an application of the $\alpha$-derivations of a multiplicative Hom-Lie conformal algebra
$(\R,\alpha)$. For any $D_\lambda\in \Cend(\R)$, define a bilinear operation $[\cdot_\lambda\cdot]_{D}$ on the vector space $\R\oplus \mathbb{R} D$ by
\begin{eqnarray}\label{4-1}
[(a+mD)_\lambda(b+nD)]_{D}=[a_\lambda b]+mD_\lambda(b)-nD_{-\lambda-\partial}(a), \ \forall \, a,b\in\R, \ m,n\in\mathbb{R},
\end{eqnarray}
and a linear map $\alpha^{'}:\R\oplus \mathbb{R} D \rightarrow \R\oplus \mathbb{R} D$ by $\alpha^{'}(a+D)=\alpha(a)+D$.
\begin{prop}
$(\R \oplus \mathbb{R} D,\alpha^{'})$ is a multiplicative Hom-Lie conformal algebra if and only if $D_\lambda$ is an $\alpha$-derivation of $(\R,\alpha)$.
\end{prop}
\begin{proof}
Suppose that $(\R \oplus \mathbb{R} D,\alpha^{'})$ is a multiplicative Hom-Lie conformal algebra. First, expanding both sides of
$\alpha^{'}([(a+mD)_\lambda(b+nD)]_{D})=[\alpha^{'}(a+mD)_\lambda\alpha^{'}(b+nD)]_{D}$ gives
$\alpha([a_\lambda b])+m\alpha\circ D_\lambda(b)-n\alpha\circ D_{-\lambda-\partial}(a)
=[\alpha(a)_\lambda\alpha(b)]+m D_\lambda\alpha(b)-n D_{-\lambda-\partial}\alpha(a)$ and thus $\alpha\circ D_\lambda=D_\lambda\circ\alpha$. Second, the Hom Jacobi identity gives
$$[\alpha^{'}(D)_\mu[a_\lambda b]_D]_D=[{[D_\mu a]_D}_{\lambda+\mu}\alpha^{'}(b)]_D+[\alpha^{'}(a)_\lambda[D_\mu b]_D]_D,$$
which is exactly $D_\mu([a_\lambda b])=[(D_\mu a)_{\lambda+\mu}\alpha(b)]+[\alpha(a)_\lambda(D_\mu b)]$ by \eqref{4-1}. Therefore, $D_\lambda$ is an $\alpha$-derivation of $(\R,\alpha)$.

Conversely, let $D_\lambda$ be an $\alpha$-derivation of $(\R,\alpha)$. For any $a, b, c\in \R$, $m,n\in\mathbb{R}$,
\begin{eqnarray*}
[(b+nD)_{-\partial-\lambda}(a+mD)]_{D}
&=&[b_{-\partial-\lambda}a]+n D_{-\lambda-\partial}(a)- m D_\lambda(b)\\
&=&-([a_\lambda b]+m D_\lambda(b)-n D_{-\lambda-\partial}(a))\\
&=&-[(a+mD)_\lambda(b+nD)]_{D},
\end{eqnarray*}
which proves \eqref{LCF2}.
And it is obvious that
\begin{eqnarray*}
&&[\partial D_\lambda a]_D=-\lambda[D_\lambda a]_D,\\
&&[\partial a_\lambda D]_D=-D_{-\partial-\lambda}(\partial a)=-\lambda[a_\lambda D]_D,\\
&&[D_\lambda\partial a]_D=D_\lambda(\partial a)=(\partial+\lambda)D_\lambda(a)=(\partial+\lambda)[D_\lambda a]_D,\\
&&[a_\lambda \partial D]_D=-(\partial D)_{-\lambda-\partial}a=(\partial+\lambda)[a_\lambda D]_D,\\
&&\alpha^{'}\circ\partial=\partial\circ\alpha^{'}.
\end{eqnarray*}
Thus \eqref{LCF1} follows. The Hom Jacobi identity is easy to check.
\end{proof}

\section{Generalized $\alpha^k$-derivations of multiplicative Hom-Lie conformal algebras}

\par

Let $(\R,\alpha)$ be a multiplicative Hom-Lie conformal algebra. Define $$\Omega=\{D_\lambda\in \Cend(\R)|D_\lambda\circ\alpha=\alpha\circ D_\lambda\}.$$
Then $\Omega$ is a Hom-Lie conformal algebra with respect to \eqref{anticom}, and $\texttt{\Der}(\R)$ is a subalgebra of  $\Omega$.

\begin{defn}\rm An element $D_\mu$ in $\Omega$ is called
\begin{itemize}
\item a {\it generalized $\alpha^{k}$-derivation} of $\R$, if there exist $D^{'}_\mu,D^{''}_\mu\in\Omega$ such that
\begin{eqnarray}\label{5-1}
[(D_\mu(a))_{\lambda+\mu}\alpha^k(b)]+[\alpha^k(a)_\lambda(D^{'}_\mu(b))]=D^{''}_\mu([a_\lambda b]), \ \forall \ a,b\in\R.
\end{eqnarray}
\item an {\it $\alpha^{k}$-quasiderivation }of $\R$, if there is $D^{'}_\mu\in\Omega$ such that
\begin{eqnarray}\label{5-2}
[(D_\mu(a))_{\lambda+\mu}\alpha^k(b)]+[\alpha^k(a)_\lambda(D_\mu(b))]=D^{'}_\mu([a_\lambda b]), \ \forall \ a,b\in\R.
\end{eqnarray}
\item an {\it $\alpha^{k}$-centroid} of $\R$, if it satisfies
\begin{eqnarray}\label{5-3}
 [(D_\mu(a))_{\lambda+\mu}\alpha^k(b)]=[\alpha^k(a)_\lambda(D_\mu(b))]=D_\mu([a_\lambda b]), \ \forall \ a,b\in\R.
\end{eqnarray}
\item an {\it $\alpha^{k}$-quasicentroid} of $\R$, if it satisfies
\begin{eqnarray}\label{5-4}
[(D_\mu(a))_{\lambda+\mu}\alpha^k(b)]=[\alpha^k(a)_\lambda(D_\mu(b))], \ \forall \ a,b\in\R.
\end{eqnarray}
\item an {\it $\alpha^{k}$-central derivation} of $\R$, if it satisfies
\begin{eqnarray}\label{5-5}
[(D_\mu(a))_{\lambda+\mu}\alpha^k(b)]=D_\mu([a_\lambda b])=0, \ \forall \ a,b\in\R.
\end{eqnarray}
\end{itemize}
\end{defn}

Denote by $\GDer_{\alpha^k}(\R)$, $\QDer_{\alpha^k}(\R)$, ${\rm C}_{\alpha^k}(\R)$, $\QC_{\alpha^k}(\R)$ and $\ZDer_{\alpha^k}(\R)$ the sets of all generalized $\alpha^{k}$-derivations,
 $\alpha^{k}$-quasiderivations, $\alpha^{k}$-centroids, $\alpha^{k}$-quasicentroids and $\alpha^{k}$-central derivations of $\R$, respectively. Set
 $$\GDer(\R):=\bigoplus_{k\geq0}\GDer_{\alpha^k}(\R),\ \, \QDer(\R):=\bigoplus_{k\geq0}\QDer_{\alpha^k}(\R).$$
$${\rm C}(\R):=\bigoplus_{k\geq0}{\rm C}_{\alpha^k}(\R), \ \ \QC(\R):=\bigoplus_{k\geq0}\QC_{\alpha^k}(\R),\ \ \ZDer(\R):=\bigoplus_{k\geq0}\ZDer_{\alpha^k}(\R).$$
It is easy to see that
\begin{eqnarray}
\ZDer(\R)\subseteq \Der(\R)\subseteq \QDer(\R)\subseteq \GDer(\R)\subseteq \Cend(\R),\,
{\rm C}(\R)\subseteq \QC(\R)\subseteq \GDer(\R).\nonumber\\ \label{tower}
\end{eqnarray}

\begin{prop}Let $(\R,\alpha)$ be a multiplicative Hom-Lie conformal algebra. Then
\begin{enumerate}
\item[$(1)$] $\GDer(\R)$, $\QDer(\R)$ and ${\rm C}(\R)$ are subalgebras of $\Omega$.
\item[$(2)$] $\ZDer(\R)$ is an ideal of $\Der(\R)$.
\end{enumerate}
\end{prop}
\begin{proof} $(1)$  We only prove that $\GDer(\R)$ is a subalgebra of $\Omega$. The proof for the other two cases is exactly analogous.

For $D_\mu\in\GDer_{\alpha^k}(\R), {H}_{\mu}\in\GDer_{\alpha^s}(\R)$, $a,b\in \R$, there exist $D^{'}_\mu,D^{''}_\mu\in\Omega$  (resp. $H^{'}_{\mu},H^{''}_{\mu}\in\Omega$ ) such that \eqref{5-1} holds for $D_\mu$ (resp. $H_{\mu}$). Recall that $\alpha'(D_\mu)=D_\mu\circ \alpha$.
\begin{eqnarray*}
[(\alpha^{'}(D_\mu)(a))_{\lambda+\mu}\alpha^{k+1}(b)]
&=&[(D_\mu(\alpha(a)))_{\lambda+\mu}\alpha^{k+1}(b)]
=\alpha([(D_\mu(a))_{\lambda+\mu}\alpha^{k}(b)])\\
&=&\alpha(D^{''}_\mu([a_\lambda b])-[\alpha^{k}(a)_\lambda D^{'}_\mu(b)])\\
&=&\alpha^{'}(D^{''}_\mu)([a_\lambda b])-[\alpha^{k+1}(a)_\lambda(\alpha^{'}(D^{'}_\mu)(b))].
\end{eqnarray*}
This gives $\alpha^{'}(D_\mu)\in\GDer_{\alpha^{k+1}}(\R)$. Furthermore, we need to show
\begin{eqnarray}\label{5-6}
[D^{''}_\mu H^{''}]_\theta([a_\lambda b])=[([D_\mu H]_\theta(a))_{\lambda+\theta}\alpha^{k+s}(b)]+[\alpha^{k+s}(a)_\lambda([D^{'}_\mu H^{'}]_\theta(b))].
\end{eqnarray}
By \eqref{anticom}, we have
\begin{eqnarray}
[([D_\mu H]_\theta(a))_{\lambda+\theta}\alpha^{k+s}(b)]
=[(D_\mu(H_{\theta-\mu}(a)))_{\lambda+\theta}\alpha^{k+s}(b)]-[(H_{\theta-\mu}(D_\mu(a)))_{\lambda+\theta}\alpha^{k+s}(b)]. \label{5-7}
\end{eqnarray}
By \eqref{5-1}, we obtain
\begin{eqnarray}
&&[(D_\mu(H_{\theta-\mu}(a)))_{\lambda+\theta}\alpha^{k+s}(b)]\nonumber\\&=&D^{''}_\mu([(H_{\theta-\mu}(a))_{\lambda+\theta-\mu}\alpha^s(b)])-[\alpha^k(H_{\theta-\mu}(a))_{\lambda+\theta-\mu}(D^{'}_\mu(\alpha^s(b)))]\nonumber\\
&=&D^{''}_\mu(H^{''}_{\theta-\mu}([a_\lambda b]))-D^{''}_\mu([\alpha^s(a)_\lambda (H^{'}_{\theta-\mu}(b)))]
\nonumber\\&&-H^{''}_{\theta-\mu}([\alpha^k(a)_\lambda(D^{'}_\mu(b))])+[\alpha^{k+s}(a)_\lambda(H^{'}_{\theta-\mu}(D^{'}_\mu(b)))],\label{5-8}\\
&&[(H_{\theta-\mu}(D_\mu(a)))_{\lambda+\theta}\alpha^{k+s}(b)]\nonumber\\&=&
H^{''}_{\theta-\mu}([(D_\mu(a))_{\lambda+\mu}\alpha^k(b)])-[\alpha^s(D_\mu(a))_{\lambda+\mu}(H^{'}_{\theta-\mu}(\alpha^k(b)))]\nonumber\\
&=&H^{''}_{\theta-\mu}(D^{''}_\mu([a_\lambda b]))-H^{''}_{\theta-\mu}([\alpha^k(a)_\lambda(D^{'}_\mu(b)])\nonumber\\
&&-D^{''}_\mu([\alpha^s(a)_\lambda(H^{'}_{\theta-\mu}(b))])+[\alpha^{k+s}(a)_\lambda(D^{'}_\mu(H^{'}_{\theta-\mu}(b))).\label{5-9}
\end{eqnarray}
Substituting \eqref{5-8} and \eqref{5-9} into \eqref{5-7} gives \eqref{5-6}. Hence $[D_\mu H]\in\GDer_{\alpha^{k+s}}(\R)[\mu]$,
and $\GDer(\R)$ is a Hom sub-algebra of $\Omega$.

(2) For ${D_1}_\mu\in\ZDer_{\alpha^k}(\R), {D_2}_{\mu}\in\Der_{\alpha^s}(\R)$, and $a, b \in \R$, we have
$$[(\alpha^{'}(D_1)_\mu(a))_{\lambda+\mu}\alpha^{k+1}(b)]=\alpha([({D_1}_\mu(a))_{\lambda+\mu}\alpha^{k}(b)])=\alpha^{'}(D_1)_\mu([a_\lambda b])=0,$$
which proves $\alpha^{'}(D_1)\in\ZDer_{\alpha^{k+1}}(\R).$ By \eqref{5-5},
\begin{eqnarray*}
[{D_{1}}_\mu D_2]_\theta([a_\lambda b])
&=&{D_1}_\mu({D_2}_{\theta-\mu}([a_\lambda b]))-{D_{2}}_{\theta-\mu}({D_{1}}_\mu([a_\lambda b]))={D_1}_\mu({D_2}_{\theta-\mu}([a_\lambda b]))\\
&=&{D_1}_\mu([({D_{2}}_{\theta-\mu}(a))_{\lambda+\theta-\mu}\alpha^{s}(b)]+[\alpha^{s}(a)_\lambda{D_{2}}_{\theta-\mu}(b)])=0,\\
{[[{D_{1}}_\mu D_2]_\theta(a)_{\lambda+\theta}\alpha^{k+s}(b)]}
&=&[({D_1}_\mu({D_2}_{\theta-\mu}(a))-{D_{2}}_{\theta-\mu}({D_{1}}_\mu(a)))_{\lambda+\theta}\alpha^{k+s}(b)]\\
&=&[-({D_{2}}_{\theta-\mu}({D_{1}}_\mu(a)))_{\lambda+\theta}\alpha^{k+s}(b)]\\
&=&-{D_2}_{\theta-\mu}([{D_1}_\mu(a)_{\lambda+\mu}\alpha^k(b)])+[\alpha^s({D_1}_\mu(a))_{\lambda+\mu}{D_{2}}_{\theta-\mu}(\alpha^k(b))]\\
&=&0.
\end{eqnarray*}
This shows that
$[{D_{1}}_\mu D_{2}] \in\ZDer_{\alpha^{k+s}}(\R)[\mu]$. Thus $\ZDer(\R)$ is an ideal of $\Der(\R)$.
\end{proof}

\begin{lem} \label{lemm5-1}
Let $(\R,\alpha)$ be a multiplicative Hom-Lie conformal algebra. Then
\begin{enumerate}
\item[$(1)$] $[\Der(\R)_\lambda{\rm C}(\R)]\subseteq{\rm C}(\R)[\lambda]$,
\item[$(2)$] $[\QDer(\R)_\lambda\QC(\R)]\subseteq\QC(\R)[\lambda]$,
\item[$(3)$] $[\QC(\R)_\lambda\QC(\R)]\subseteq\QDer(\R)[\lambda]$.
\end{enumerate}
\end{lem}
\begin{proof} It is straightforward. \end{proof}

\begin{thm}
Let $(\R,\alpha)$ be a multiplicative Hom-Lie conformal algebra. Then $$\GDer(\R)=\QDer(\R)+\QC(\R).$$
\end{thm}
\begin{proof} For $D_\mu\in\GDer_{\alpha^k}(\R)$, there exist $D^{'}_\mu, D^{''}_\mu\in\Omega$ such that
\begin{eqnarray}\label{5-10}
[D_\mu(a)_{\lambda+\mu}\alpha^k(b)]+[\alpha^k(a)_\lambda D^{'}_\mu(b)]=D^{''}_\mu([a_\lambda b]), \forall \, a,b\in \R.
\end{eqnarray}
By \eqref{LCF2} and (\ref{5-10}), we get
\begin{eqnarray}
[\alpha^k(b)_{-\partial-\lambda-\mu}D_\mu(a)]+[D^{'}_\mu(b)_{-\partial-\lambda}\alpha^k(a)]=D^{''}_\mu([b_{-\partial-\lambda}a]).\label{represen6}
\end{eqnarray}
By \eqref{LCF1} and setting $\lambda^{'}=-\partial-\lambda-\mu$ in (\ref{represen6}), we obtain
\begin{eqnarray}
[\alpha^k(b)_{\lambda^{'}}D_\mu(a)]+[D^{'}_\mu(b)_{\mu+\lambda^{'}}\alpha^k(a)]=D^{''}_\mu([b_{\lambda^{'}}a]).\label{5-11}
\end{eqnarray}
Then, changing the place of $a$ and $b$ and replacing $\lambda^{'}$ by $\lambda$ in (\ref{5-11}) give
\begin{eqnarray}
[\alpha^k(a)_\lambda D_\mu(b)]+[D^{'}_\mu(a)_{\lambda+\mu}\alpha^k(b)]=D^{''}_\mu([a_\lambda b]).\label{5-12}
\end{eqnarray}
Combining (\ref{5-10}) with (\ref{5-12}) gives
\begin{eqnarray*}
&&[\frac{D_\mu+D^{'}_\mu}{2}(a)_{\lambda+\mu}\alpha^k(b)]+[\alpha^k(a)_{\lambda}\frac{D_\mu+D^{'}_\mu}{2}(b)]=D^{''}_\mu([a_\lambda b]),\\
&&[\frac{D_\mu-D^{'}_\mu}{2}(a)_{\lambda+\mu}\alpha^k(b)]-[\alpha^k(a)_{\lambda}\frac{D_\mu-D^{'}_\mu}{2}(b)]=0.
\end{eqnarray*}
It follows that $\frac{D_\mu+D^{'}_\mu}{2}\in\QDer_{\alpha^k}(\R)$ and $\frac{D_\mu-D^{'}_\mu}{2}\in\QC_{\alpha^k}(\R)$. Hence
$$D_\mu=\frac{D_\mu+D^{'}_\mu}{2}+\frac{D_\mu-D^{'}_\mu}{2}\in\QDer(\R)+\QC(\R),$$
proving that $\GDer(\R)\subseteq\QDer(\R)+\QC(\R).$  The reverse inclusion relation follows from \eqref{tower} and Lemma \ref{lemm5-1}.
\end{proof}

\begin{thm}
Let $(\R,\alpha)$ be a multiplicative Hom-Lie conformal algebra, $\alpha$ a surjection and $\Z(\R)$ the center of $\R$. Then
$[{\rm C}(\R)_\lambda\QC(\R)]\subseteq\Chom(\R,\Z(\R))[\lambda]$. Moreover, if $\Z(\R)=0$, then $[{\rm C}(\R)_\lambda\QC(\R)]=0$.
\end{thm}
\begin{proof}  Since $\alpha$ is surjective, for any $b^{'}\in \R,$ there exists $b \in \R$ such that $b^{'}=\alpha^{k+s}(b)$.
For ${D_1}_\mu\in{\rm C}_{\alpha^k}(\R),{D_2}_{\mu}\in\QC_{\alpha^s}(\R)$, and $a\in \R$,
by \eqref{5-3}and \eqref{5-4}, we have
\begin{eqnarray*}
[([{D_{1}}_\mu D_{2}]_\theta(a))_{\lambda+\theta}b^{'}]
&=&[([{D_{1}}_\mu D_{2}]_\theta(a))_{\lambda+\theta}\alpha^{k+s}(b)]\\
&=&[({D_1}_\mu({D_2}_{\theta-\mu}(a)))_{\lambda+\theta}\alpha^{k+s}(b)]
-[({D_{2}}_{\theta-\mu}({D_{1}}_\mu(a)))_{\lambda+\theta}\alpha^{k+s}(b)]\\
&=&{D_1}_\mu([{D_2}_{\theta-\mu}(a)_{\lambda+\theta-\mu}\alpha^{s}(b)])
-[\alpha^s({D_1}_\mu(a))_{\lambda+\mu}{D_2}_{\theta-\mu}(\alpha^k(b))]\\
&=&{D_1}_\mu([{D_2}_{\theta-\mu}(a)_{\lambda+\theta-\mu}\alpha^{s}(b)])
-{D_1}_\mu([\alpha^s(a)_{\lambda}{D_2}_{\theta-\mu}(b)])\\
&=&{D_1}_\mu([{D_2}_{\theta-\mu}(a)_{\lambda+\theta-\mu}\alpha^{s}(b)]
-[\alpha^s(a)_{\lambda}{D_2}_{\theta-\mu}(b)])\\
&=&0.
\end{eqnarray*}
Hence $[{D_{1}}_\mu D_{2}](a)\in \Z(\R)[\mu]$, and then $[{D_{1}}_\mu D_{2}]\in\Chom(\R,\Z(\R))[\mu]$. If $\Z(\R)=0$, then $[{D_{1}}_\mu D_{2}](a)=0$, $\forall$ $a\in\R.$ Thus $[{\rm C}(\R)_\lambda\QC(\R)]=0$.
\end{proof}

\begin{prop}
Let $(\R,\alpha)$ be a multiplicative Hom-Lie conformal algebra, and $\alpha$ a surjection. If $\Z(\R)=0$, then $\QC(\R)$ is a Hom-Lie conformal algebra if and  only if $[\QC(\R)_\lambda\QC(\R)]=0$.
\end{prop}
\begin{proof}
$(\Rightarrow)$  Assume that $\QC(\R)$ is a Hom-Lie conformal algebra. Since $\alpha$ is surjective,
for any $b^{'}\in \R$, there exists $b \in \R$ such that $b^{'}=\alpha^{k+s}(b)$. For ${D_1}_\mu\in\QC_{\alpha^k}(\R),{D_{2}}_{\mu}\in\QC_{\alpha^s}(\R)$, $[{D_{1}}_\mu D_{2}]\in\QC_{\alpha^{k+s}}(\R)[\mu]$. For $a\in \R$, by \eqref{5-4}, we have
\begin{eqnarray}\label{5-15}[([{D_{1}}_\mu D_{2}]_\theta(a))_{\lambda+\theta}b^{'}]=[([{D_{1}}_\mu D_{2}]_\theta(a))_{\lambda+\theta}\alpha^{k+s}(b)]
=[\alpha^{k+s}(a)_\lambda([{D_{1}}_\mu D_{2}]_\theta(b))].\end{eqnarray}
By \eqref{anticom} and \eqref{5-4}, we obtain
\begin{eqnarray}\label{5-16}
&&[([{D_{1}}_\mu D_2]_\theta(a))_{\lambda+\theta}\alpha^{k+s}(b)]\nonumber\\
&=&[({D_1}_\mu({D_2}_{\theta-\mu}(a)))_{\lambda+\theta}\alpha^{k+s}(b)]-[({D_2}_{\theta-\mu}({D_1}_\mu(a)))_{\lambda+\theta}\alpha^{k+s}(b)]\nonumber\\
&=&[\alpha^k({D_2}_{\theta-\mu}(a))_{\lambda+\theta-\mu}({D_1}_\mu(\alpha^s(b)))]
-[\alpha^s({D_1}_{\mu}(a))_{\lambda+\mu}({D_2}_{\theta-\mu}(\alpha^{k}(b)))]\nonumber\\
&=&[\alpha^{k+s}(a)_\lambda({D_2}_{\theta-\mu}({D_1}_\mu(b)))]-[\alpha^{k+s}(a)_\lambda({D_1}_\mu({D_2}_{\theta-\mu}(b)))]\nonumber\\
&=&-[\alpha^{k+s}(a)_\lambda([{D_{1}}_\mu D_{2}]_\theta(b))].
\end{eqnarray}
Combining \eqref{5-15} with \eqref{5-16} gives
$$[([{D_{1}}_\mu D_{2}]_\theta(a))_{\lambda+\theta}b^{'}]=[([{D_{1}}_\mu D_2]_\theta(a))_{\lambda+\theta}\alpha^{k+s}(b)]=0,$$
and thus $[{D_{1}}_\mu D_{2}]_\theta(a)\in\Z(\R)[\mu]=0$, since $\Z(\R)=0$.
Therefore, $[{D_{1}}_\mu D_{2}]=0$.

$(\Leftarrow)$ It is clear.
\end{proof}

Let $(\R,\alpha)$ be a multiplicative Hom-Lie conformal algebra and $t$ an indeterminate. Denote
\begin{eqnarray*}
\breve{\R}=\R[t\mathbb{C}(t)/(t^3)]=\{\sum(a\otimes t+b\otimes t^2)|a,b\in \R\}.
 \end{eqnarray*}
 Define on $\breve{\R}$
\begin{eqnarray*}
\breve{\alpha}(a\otimes t^i) =\alpha(a)\otimes t^i, \ \
\partial(a\otimes t^i)=\partial a\otimes t^i, \ i=1, 2.
 \end{eqnarray*}
Then $(\breve{\R},\breve{\alpha})$ is a Hom-Lie conformal algebra with the following $\lambda$-bracket
\begin{eqnarray}
[(a\otimes t^i)_\lambda(b\otimes t^j)]=[a_\lambda b]\otimes t^{i+j}, \ \ i,\,j=1,\, 2.
\end{eqnarray}
In the following, we shall simply write $xt^i$ for $a\otimes t^i$, $i=1,2$, and denote by $[\R,\R]$ as the $\mathbb{C}$-linear
span of all $\lambda$-coefficients in the products $[a_\lambda b]$, where $a, b \in\R$.
If $U$ is a subspace of $\R$ such that $\R=U\oplus[\R,\R]$, then
$\breve{\R}=\R t+\R t^2=\R t+[\R,\R]t^2+Ut^2.$

Define a map $\varphi:\QDer(\R)\rightarrow\Cend(\breve{\R})$ by
\begin{eqnarray}\label{5-20}
\varphi(D_\mu)(at+bt^2+ut^2)=D_\mu(a)t+D^{'}_\mu(b)t^2, \forall \ a\in\R, \ b\in [\R,\R], \ u\in U,
\end{eqnarray}
where $D_\mu,D^{'}_\mu$ satisfy \eqref{5-2}.

\begin{prop}
\begin{enumerate}
\item[$(1)$] $\varphi$ is injective and $\varphi(D_\mu)$ does not depend on the choice of $D^{'}_\mu$,
\item[$(2)$] $\varphi(\QDer(\R))\subseteq\Der(\breve{\R})$.
\end{enumerate}
\end{prop}
\begin{proof}
(1) If $\varphi({D_1}_\mu)=\varphi({D_2}_\mu)$, then
$$\varphi({D_1}_\mu)(at+bt^2+ut^2)=\varphi({D_2}_\mu)(at+bt^2+ut^2), \ \forall \ a\in\R, \ b\in [\R,\R], \ u\in U,$$
that is $${D_1}_\mu(a)t+{D^{'}_1}_\mu(b)t^2={D_2}_\mu(a)t+{D^{'}_2}_\mu(b)t^2.$$ Thus ${D_1}_\mu(a)={D_2}_\mu(a)$, $\forall\ a\in\R$. Then ${D_1}_\mu={D_2}_\mu$ and thus $\varphi$ is injective.

If there exists another $D^{''}_\mu$ satisfying \eqref{5-20}.
Since both $D^{'}_\mu$ and $D^{''}_\mu$ satisfy \eqref{5-2}, we have $D^{'}_\mu([c_\lambda d])=D^{''}_\mu([c_\lambda d])$ for any $c,d\in\R $, namely, $D^{'}_\mu(b)=D^{''}_\mu(b)$, $\forall\ b\in[\R,\R]$.
Hence $$\varphi(D_\mu)(at+bt^2+ut^2)=D_\mu(a)t+D^{'}_\mu(b)t^2=D_\mu(a)t+D^{''}_\mu(b)t^2,$$
which implies $\varphi(D_\mu)$ does not depend on the choice of $D^{'}_\mu$.

(2) Note that $[(at^i)_\lambda(bt^j)]=[a_\lambda b]t^{i+j}=0$, for all $i+j\geq3$. For $D_\mu\in\QDer_{\alpha^k}(\R)$, we only need to show
\begin{eqnarray}\label{5-21}
\varphi(D_\mu)([(at)_\lambda (bt)])=[\varphi(D_\mu)(at)_{\lambda+\mu}\breve{\alpha}^k(bt)]+[\breve{\alpha}^k(at)_\lambda\varphi(D_\mu)], \ \forall\, a,b\in \R.
\end{eqnarray}
Indeed, we have
\begin{eqnarray*}
\varphi(D_\mu)([at_\lambda bt])
&=&\varphi(D_\mu)([a_\lambda b]t^2)=D^{'}_\mu([a_\lambda b])t^2\\
&=&([D_\mu(a)_{\lambda+\mu}\alpha^k(b)]+[\alpha^k(a)_\lambda D_\mu(b)])t^2\\
&=&[(D_\mu(a)t)_{\lambda+\mu}(\alpha^k(b)t)]+[(\alpha^k(a)t)_\lambda (D_\mu(b)t)]\\
&=&[\varphi(D)_\mu(at)_{\lambda+\mu}\breve{\alpha}^k(bt)]+[\breve{\alpha}^k(at)_\lambda(\varphi(D)_\mu(bt))],
\end{eqnarray*}
which proves \eqref{5-21} and thus $\varphi(D_\mu)\in\Der_{\alpha^k}(\breve{\R})$.
\end{proof}

\end{document}